\documentclass[11pt]{amsart}

\usepackage{amsfonts}
\usepackage[]{amsmath}
\usepackage[]{amssymb}
\usepackage[]{amsthm}
\usepackage[]{amsfonts}

\newtheorem  {theorem}            {Theorem}
\newtheorem  {proposition}            {Proposition}
\newtheorem  {lemma}              {Lemma}

\newtheorem  {remark}             {Remark}
\newtheorem {remark*}            {Remark}

\newtheorem* {conjecture*}        {Conjecture}
\newtheorem* {theorem*}           {Theorem}
\newtheorem* {acknowledgements*}  {Acknowledgements}

\newcommand {\Ric}  {\operatorname{Ric}}
\newcommand {\R}    {\operatorname{R}}

\renewcommand {\div} {\operatorname {div}}

\newcommand {\tr} {\operatorname {tr}}

\usepackage{latexsym, amssymb}

\newcommand{\be}{\begin{equation}}
\newcommand{\ee}{\end{equation}}

\def\p{\partial}

\def\R{\mathbb{R}}

\def\tr{{\rm tr}}

\def\p{\partial}
\def\lf{\left}
\def\ri{\right}

\def\R{\Bbb R}

\def\Ric{\text{\rm Ric}}
\def\div{\text{\rm div}}

\def\Pi{\displaystyle{\mathbb{II}}}

\begin{document}

\title{Extension of a theorem of Shi and Tam}

\begin{abstract}
In this note, we prove the following generalization of a theorem of Shi and Tam \cite{ShiTam02}: Let $(\Omega, g)$ be an $n$-dimensional ($n \geq 3$) compact Riemannian manifold, spin when $n>7$, with non-negative scalar curvature and mean convex boundary. If every boundary component $\Sigma_i$ has
positive scalar curvature and embeds isometrically as a mean convex star-shaped hypersurface ${\hat \Sigma}_i \subset \R^n$, then 
\begin{equation*} 
 \int_{\Sigma_i} H \ d \sigma \le \int_{{\hat \Sigma}_i} \hat{H} \ d {\hat \sigma}
\end{equation*}
where $H$ is the mean curvature of $\Sigma_i$ in $(\Omega, g)$, $\hat{H}$ is the Euclidean mean curvature of ${\hat \Sigma}_i$ in  $\R^n$, and where $d \sigma$ and $d {\hat \sigma}$ denote the respective volume forms. Moreover, equality 
holds for some boundary component $\Sigma_i$ if, and only if, $(\Omega, g)$ is isometric to a domain in $\R^n$.
 
In the proof, we make use of a foliation of the exterior of the $\hat \Sigma_i$'s in $\R^n$ by the $\frac{H}{R}$-flow studied by Gerhardt \cite{Gerhardt90} and Urbas \cite{Urbas90}. We also carefully establish the rigidity statement in low dimensions without the spin assumption that was used in \cite{ShiTam02}. 
\end{abstract}

\author{Michael Eichmair$^1$, Pengzi Miao$^2$, and Xiaodong Wang$^3$}

\address{Department of Mathematics, MIT, 77 Massachusetts Avenue, Cambridge, MA 02139, USA}
\email{eichmair@math.mit.edu}

\address{Department of Mathematics, University of Miami, Coral Gables, FL 33146, USA;
School of Mathematical Sciences, Monash University, Victoria, 3800, Australia}
\email{p.miao@math.miami.edu,
Pengzi.Miao@sci.monash.edu.au}

\address{Department of Mathematics, Michigan State University,
East Lansing, MI 48824, USA}
\email{xwang@math.msu.edu}

\thanks{$^1$ Research partially supported by Australian Research Council Discovery Grant \#DP0987650 and by the NSF grant DMS-0906038. }
\thanks{$^2$ Research partially supported by Australian Research Council Discovery Grant \#DP0987650.}
\thanks{$^3$ Research partially supported by the NSF grant DMS-0905904.}

\date{}

\maketitle

\markboth{Michael Eichmair, Pengzi Miao, and Xiaodong Wang}
{A generalization of a theorem of Shi and Tam}

\section{Introduction and statement of results}

In the work of Shi and Tam \cite[Theorem 4.1]{ShiTam02}, the positive mass theorem was used in a novel way to yield beautiful results on the boundary behavior of compact Riemannian manifolds with non-negative scalar curvature:

\begin{theorem} {\rm (Shi-Tam)} \label{thm: ShiTam} Let $(\Omega, g)$ be an $n$-dimensional compact Riemannian spin manifold with non-negative scalar curvature and mean convex boundary. 
If every component $\Sigma_i$ of the boundary is isometric to a strictly convex hypersurface ${\hat \Sigma}_i \subset \R^n$, then 
\begin{equation} 
\label{eqn: main theorem} \int_{\Sigma_i} H \ d \sigma \le \int_{{\hat \Sigma}_i} \hat{H} \ d {\hat \sigma}
\end{equation}
where $H$ is the mean curvature of $\Sigma_i$ in $(\Omega, g)$, $\hat{H}$ is the Euclidean mean curvature of ${\hat \Sigma}_i$ in  $\R^n$, and where $d \sigma$ and $d {\hat \sigma}$ denote the respective volume forms. Moreover, equality in (\ref{eqn: main theorem}) holds for some boundary component $\Sigma_i$ if, and only if, $(\Omega, g)$ is isometric to a domain in $\R^n$.
\end{theorem}

\noindent Note that every compact strictly convex hypersurface ${\hat \Sigma}$ as above bounds a compact region in $\R^n$. Throughout this paper, the (scalar) mean curvature of the boundary of a set is computed with respect to the outward unit normal. With this convention, the boundary of a unit ball in $\R^n$ has mean curvature $(n-1)$. As a consequence of the Weyl embedding problem solved by Pogorelov \cite{Pog} and Nirenberg \cite{Nirenberg} independently, in dimension $n = 3$ the assumption in Theorem \ref{thm: ShiTam} that every boundary component $\Sigma_i$ embeds isometrically as a strictly convex hypersurface in $\R^3$ is equivalent to the requirement that the boundary of $(\Omega, g)$ has positive Gaussian curvature. Hence,  the conclusion (\ref{eqn: main theorem}) in the theorem shows that positively curved mean convex boundaries in time-symmetric initial data sets 
satisfying the dominant energy condition
have non-negative Brown-York mass \cite{BrownYork93}. This result has been generalized to subsets of general initial data sets by M. Liu and S.-T. Yau in \cite{LiuYau03} \cite{LiuYau06}, and to a hyperbolic setting by M.-T. Wang and S.-T. Yau in \cite{WangYau07}. 
In dimensions $n>3$, there are no analogous intrinsic conditions on the boundary of $(\Omega, g)$ that guarantee that its components embed isometrically into $\R^n$. 

There are two major ingredients in Shi and Tam's proof of Theorem \ref{thm: ShiTam}. 
For simplicity, let us assume that the boundary of $\Omega$ has only one component. Let
$\iota: \Sigma:=\partial\Omega\rightarrow\mathbb{R}^{n}$ be its isometric
embedding. Let $\nu:$ $\iota(\Sigma)\rightarrow\mathbb{S}^{n-1}$ be the outer unit
normal. Since $\iota (\Sigma)$ is assumed to be a strictly convex
hypersurface in $\mathbb{R}^{n}$ there is a smooth family of embeddings
\[
F: \Sigma \times [0, \infty] \rightarrow\mathbb{R}^{n} \text{ where } F_{t}\left(  \sigma \right)
=F(\sigma, t )=\iota \left(  \sigma \right)  +t\nu\left( \iota(\sigma)\right).
\]
Note that $F_{t} (\Sigma)$ are the `outer' distance surfaces of $\iota (\Sigma)$.
If $\hat \Omega$ denotes the bounded domain enclosed by $\iota (\Sigma)$, then $\{F_t(\Sigma)\}_{t\geq0}$ foliates $\mathbb{R}^{n}\backslash\hat{\Omega}$ and the Euclidean metric on this set can be written as
\[
\mathcal{G}=dt^{2}+g_{t},
\]
where $g_{t}$ is the first fundamental form of the embedding $F_{t}
:\Sigma\rightarrow\mathbb{R}^{n}$. \ Shi and Tam then find an asymptotically flat scalar flat
metric 
\[
\widetilde{\mathcal{G}}=u^{2}dt^{2}+g_{t}
\]
on $\mathbb{R}^{n}\backslash\hat{\Omega}$ such that the mean curvature of $\Sigma \times \{0\}$ in 
$(\mathbb{R}^{n}\backslash\hat{\Omega}, \widetilde{\mathcal{G}})$
 coincides with the mean curvature of $\Sigma$ in $(\Omega, g)$. The function 
 $u: \Sigma \times [0, \infty) \rightarrow (0, \infty)$ is obtained as the solution of a certain nonlinear heat equation. 

The second crucial ingredient of their proof is the observation that
\begin{equation}\label{totalH}
m(t)=\int_{\Sigma\times\{t\}}{H}_{t}\left(  1-u^{-1}\right)  d\sigma_{t},
\end{equation}
is a non-increasing function of $t \geq 0$ where ${H}_{t}$ is the mean curvature of $F_{t} (\Sigma)$
and $d\sigma_{t}$ is its volume form, and that the limit of $m(t)$ as $t\rightarrow\infty$ is the ADM mass of $\widetilde{\mathcal{G}}$. The final step
in their proof is to apply the positive mass theorem for spin manifolds to the
asymptotically flat manifold obtained from gluing $(\Omega, g)$ and 
$(\mathbb{R}^{n}\backslash\hat{\Omega},\widetilde{\mathcal{G}})$
along their boundaries.

Given Theorem \ref{thm: ShiTam}, a natural question to ask is whether the requirement that the embeddings of the boundary components be strictly convex is really necessary. In this paper we present some
variations on Shi-Tam's method. We prove that their theorem continues to hold provided each boundary component can be embedded isometrically into $\R^n$ in such a way that the unbounded component of the complement of the embedded surface is foliated by mean convex leaves of positive scalar curvature. Moreover, we note that the spin assumption can be dropped when $3\leq n\leq 7$. To be precise, we have the following

\begin{proposition} \label{prop: main} The conclusion of Theorem \ref{thm: ShiTam} remains valid if the spin assumption is dropped in dimensions $3 \leq n \leq 7$ and if the hypothesis that every boundary component can be embedded isometrically as a strictly convex hypersurface in $\R^n$ is eased to the requirement that every boundary component is isometric to an embedded surface ${\hat \Sigma_i} \subset \R^n$ that  can be deformed to a strictly convex hypersurface via an `expanding' flow through embedded mean convex hypersurfaces of positive scalar curvature. That is, if there exists a smooth map 
$ F : {\Sigma_i} \times [0, 1] \rightarrow \R^n $
such that $ F(\Sigma_i, 0) = {\hat \Sigma_i} $, $F({\Sigma_i}, t)$ is a smooth embedded, mean convex hypersurface with positive scalar curvature for every $t \in [0, 1]$, $\frac{ \p F}{\p t} = \eta \nu$ where $ \nu$ is the outward unit normal of $F( {\Sigma_i}, t)$ and where $\eta >0$ is a smooth positive function, and $F({\Sigma_i}, 1) \subset \R^n$ is a strictly convex hypersurface. 
\end{proposition}

This proposition makes it possible to try some other foliations even when the embeddings are not convex. In particular, if each $ \hat{\Sigma}_i $ 
is a star-shaped surface with positive scalar curvature and positive mean curvature, we can use the result of Gerhardt \cite{Gerhardt90} and Urbas \cite{Urbas90} to obtain 
a foliation of $\mathbb{R}^{n} \backslash \hat{\Omega}_i$. This leads to

\begin{theorem} \label{thm: main theorem} The conclusion of Theorem \ref{thm: ShiTam} remains valid if the assumption that every boundary component embeds as a strictly convex hypersurface in $\R^n$ is relaxed to the requirement that the boundary of $(\Omega, g)$ has positive scalar curvature and that each boundary component is isometric to a mean-convex, star-shaped hypersurface in $\R^n$. Moreover, the spin assumption can be dropped in dimensions $3 \leq n \leq 7$.
\end{theorem}

This paper is organized as follows.  In section \ref{BST-construction}, we discuss the generalized quasi-spherical metric construction of Bartnik \cite{Bartnik93}, Smith-Weinstein \cite{SmithWeinstein04}, and Shi-Tam \cite{ShiTam02} 
along any finite-time geometric flow in the generality that is appropriate for our needs (see also \cite{Neves}, where this construction is carried out along a foliation of inverse mean curvature flow), and we derive the crucial monotonicity of  \eqref{totalH}. In section \ref{pfofthm}, we prove Proposition \ref{prop: main} with the spin assumption, and point out that the inequality in Proposition \ref{prop: main} remains true without  the spin assumption.  In section \ref{sec: rigidity}, we examine carefully the equality case in Proposition \ref{prop: main} without the spin assumption. It appears that the rigidity case of the positive mass theorem \cite{SchoenYau79} and \cite{Schoen87} on non-spin asymptotically flat manifolds \emph{with Lipschitz singularities} as in \cite{ShiTam02}, \cite{Miao02} is not addressed in the literature, so we derive it carefully in section \ref{sec: rigidity}. In section \ref{Gerhardt-Urbas}, we  review the works of Gerhardt \cite{Gerhardt90} and Urbas \cite{Urbas90}, which together with Proposition \ref{prop: main}, imply Theorem \ref{thm: main theorem}.

\section{Generalized quasi-spherical metric construction and the monotonicity of the Brown-York mass}

In this section we review the results in \cite{Bartnik93}, \cite{SmithWeinstein04}, \cite{ShiTam02} on the construction of certain asymptotically flat, scalar flat metrics defined outside smooth compact sets $\hat \Omega \subset \R^n$ such that both boundary metric and mean curvature are prescribed. Using geometric arguments, we re-derive the necessary evolution equations contained in these references, as well as the required a priori estimate. We discuss how this construction is used in our proof of Proposition \ref{prop: main}. The monotonicity of the Brown-York mass is derived in Proposition \ref{prop: monontone} and generalizes the result in \cite[Lemma 4.2]{ShiTam02}. 

\label{BST-construction}  
Let $ \Sigma $ be a closed $(n-1)$-dimensional manifold
and let $t_0 > 0$. Let $\{g_{t}\}_{t \in [0, t_0]}$ 
be a smooth family of Riemannian metrics on $\Sigma$. 
Given any smooth positive function $ f $ on  $\Sigma \times [0, t_0]$, 
consider the Riemannian metric 
\begin{equation}
g_f := f^2 dt^2 + g_{t}. 
\end{equation}
Let $H_f$, $h_f$ denote the mean curvature and second fundamental form of $\Sigma \times \{ t \}$ with respect to the $g_f$ metric (where the mean curvature is the trace of the second fundamental form, and the signs are so that $H_f$ measures the rate of change of the area element in direction $\partial_t$).
If we let 1 denote the constant function that is identically one, then  $ f h_f = h_1 $ 
and $ f H_f = H_1 $.  

As in Proposition \ref{prop: main} we make the following standing

{\bf Assumption: }  The scalar curvature $ R (g_{t})=:2K $ of $ g_t $ and 
the mean curvature $H_1$ of the leaves $\Sigma \times \{t\}$ with respect to $ g_1 $ are everywhere positive.

\begin{proposition} [cf. \cite{Bartnik93}, \cite{SmithWeinstein04}, \cite{ShiTam02}] \label{prop-warped}
Under the above assumption, given any positive function $u_0$ on $ \Sigma \times \{0\}$, 
there is a smooth positive function $ u $ on $ \Sigma \times [0, t_0]$
such that  the scalar curvature $ R (g_u) $ of $ g_u $ is identically zero
and $ u|_{ t = 0 } = u_0 $.
\end{proposition}

\begin{proof} 
By the Jacobi equation and the Gauss equation,  we have that 
\begin{eqnarray} \label{eqofdH}
\frac{d}{dt} H_u &=& \frac{d}{dt} \frac{H_1}{u} \nonumber \\ 
&=& - \Delta u - u (|h_u|^2 + Ric(g_u) (\nu_u, \nu_u)) \nonumber \\ 
&=& - \Delta u - \frac{u}{2} (R (g_u) - 2 K + H_u^2 + |h_u|^2) \nonumber \\ 
&=& - \Delta u + K u - \frac{1}{2u} (H_1^2 + |h_1|^2)
\end{eqnarray}

\noindent where $ | \cdot |^2 $ is taken with respect to $ g_{t}$,
$ Ric(g_u) $ denotes the Ricci curvature of $ g_u $, $ \nu_u = u^{-1} \partial_t $, 
and  $\Delta$ is taken with respect to $g_{t}$. Setting $ R ( g_u ) = 0 $,  we have that 
\begin{eqnarray} \label{eq: basic prescribed scalar curvature equation}
- u' + \frac{u^2}{H_1} \Delta u = u^3 \frac{K}{H_1} - \frac{u}{2 H_1} (H_1^2 + |h_1|^2 + 2 H_1')
\end{eqnarray}
where we abbreviated the $t$-derivative with a dash. 

The assumptions $u_0, H_1 > 0 $ guarantee that  \eqref{eq: basic prescribed scalar curvature equation} 
has a smooth positive solution  $ u $  with initial condition $ u|_{t=0} = u_0 $ on some small
interval $[0,  \delta)$, $ \delta > 0$. To show such a solution $ u $ can be extended to the whole interval $[0, t_0]$,
it suffices to prove that $ u $ remains bounded from above and from below by some positive constants
(depending only on $ t_0$) by standard parabolic theory.

To derive an upper bound for $ u$, let $ C > \max_\Sigma u_0  $ be a positive constant such that 
\begin{equation}
C^2 \geq \max_{\Sigma \times [0, t_0]} \frac{H_1 ^2 + |h_1|^2 + 2 H_1'}{2 K} .
\end{equation}
Suppose $ u \geq C $ somewhere on $ \Sigma \times [0, \delta)$. Since $ u (\sigma, 0) < C $,
there exists $( \tilde{\sigma}, \tilde{t}) \in \Sigma \times (0, \delta)$ such that
$ u( \tilde{\sigma}, \tilde{t}) = C $ and $ u(\sigma, t) \leq C $ for $ t \leq \tilde{t}$. 
At $(\tilde{\sigma}, \tilde{t})$,  by the assumptions $H_1 > 0$ and $ K > 0 $, we have
\begin{equation*}
 - u' + \frac{u^2}{H_1} \Delta u  \leq 0 , \ \ 
u^3 \frac{K}{H_1} - \frac{u}{2 H_1} (H_1^2 + |h_1|^2 + 2 H_1') > 0,
\end{equation*}
thus a contradiction to \eqref{eq: basic prescribed scalar curvature equation}. Hence, $ u < C $. 

To get a lower bound of $ u$, define $\underline u = \beta e^{- \gamma t}$ where
 $ \beta  <  \min_\Sigma u_0 $ is a positive constant and $ \gamma $  is another 
 positive constant such that
\begin{eqnarray} 
\gamma >  \max_{\Sigma \times [0, t_0]} \left( \beta^2 \frac{K}{ H_1} - \frac{H_1^2 + |h_1|^2 + 2 H_1'}{ 2 H_1}
\right) . 
\end{eqnarray} 
Let $ v = u - \underline u$, then $ v$ satisfies
\begin{eqnarray} \label{eqofv}
- v' + \frac{u^2}{H_1} \Delta v <  ( u^3 - \underline u^3) \frac{K}{H_1} - \frac{ v  }{2 H_1} (H_1^2 + |h_1|^2 + 2 H_1') .
\end{eqnarray}
Now suppose $ v \leq 0 $ somewhere on $ \Sigma \times [0, \delta)$. Since $ v |_{ t =  0} > 0 $,
there exists $( \tilde{\sigma}, \tilde{t}) \in \Sigma \times (0, \delta)$ such that
$ v( \tilde{\sigma}, \tilde{t}) = 0 $ and $ v \geq 0 $ for $ t \leq \tilde{t}$. 
At $(\tilde{\sigma}, \tilde{t})$,  by the assumptions $H_1 > 0$ and $ K > 0 $, we have
\begin{equation*}
 - v' + \frac{u^2}{H_1} \Delta v  \geq 0 , \ \ 
(u^3 - \underline u^3) \frac{K}{H_1} - \frac{ v }{2 H_1} (H_1^2 + |h_1|^2 + 2 H_1')  =  0,
\end{equation*}
thus a contradiction to \eqref{eqofv}. Hence $ v  > 0 $, i.e. $ u(\sigma, t) > \beta e^{- \gamma t}$ for all $(\sigma, t) \in \Sigma \times [0, t_0)$. 
\end{proof}

The monotonicity formula in the following proposition generalizes \cite[Lemma 4.2]{ShiTam02}. 
 
\begin{proposition} \label{prop: monontone}
Suppose $ u $ and $ \eta $ are two smooth positive functions on $ \Sigma \times [0, t_0]$
such that $ R( g_u) = 0 $ and $ Ric(g_\eta) = 0 $. 
Then
$$ \int_{\Sigma \times \{t\}} (H_\eta - H_u ) \ d \sigma_t $$
is monotone non-increasing in $ t $. Here $ d \sigma_t $ denotes the volume form
of $ g_{t} $ on $ \Sigma \times \{t\}$. 
\end{proposition}

\begin{proof}
By \eqref{eqofdH}, we have
\begin{equation} \label{dtoftotalH}
\begin{split}
& \ \frac{d}{dt} \left( \int_{\Sigma \times \{t\}} (H_\eta - H_u ) \ d \sigma_t  \right) \\
 = & \ \int_{\Sigma \times \{t\}} (\eta^{-1} - u^{-1} ) H_1^2 +  K( \eta - u) - \frac{1}{2}(\eta^{-1} - u^{-1} )(H_1^2 + |h_1|^2)
  \ d \sigma_t .  
  \end{split}
\end{equation}
By the Gauss equation and the assumption that  $ Ric(g_\eta) = 0 $, we have
\begin{equation} \label{gaussineta}
2 K =  H_\eta^2 - | h_\eta|^2 = \eta^{-2} ( H_1^2 - |h_1|^2). 
\end{equation}
Therefore, it follows from \eqref{dtoftotalH} and \eqref{gaussineta} that 
\begin{equation}
\ \frac{d}{dt} \left( \int_{\Sigma \times \{t\}} (H_\eta - H_u ) \ d \sigma_t  \right) 
=  - \int_{\Sigma \times \{t\}} K ( \eta - u )^2 u^{-1} \ d \sigma_t \leq 0,
\end{equation}
where we also used the assumption that $ K > 0 $.  
\end{proof}

\section{Proof of Proposition \ref{prop: main} using the spin assumption for the equality case} \label{pfofthm}

Let $(\Omega, g)$ be as in Proposition \ref{prop: main}. Let $\{\Sigma_1, \ldots, \Sigma_p\}$ be the boundary components of $\partial \Omega$, and fix one $\Sigma_i$. The pull-back of the Euclidean metric by $F$ to $\Sigma_i \times [0, 1]$ has the form $g_{\eta_i} = \eta_i^2 dt^2 + g_t$ where $g_{t}$ is a time dependent metric on $\Sigma_i$. Using Proposition \ref{prop-warped} we can find a smooth positive function $u_i: \Sigma_i \times [0, 1] \to \R$ so that the scalar curvature of the metric $g_{u_i} := u_i^2 dt^2 + g_t$ vanishes and so that $u_i(\sigma, 0) = \eta(\sigma, 0) \frac{\iota^* \hat H}{H}$. Here,  $\hat H$ is the Euclidean mean curvature of $\hat \Sigma_i = \iota (\Sigma_i) \subset \R^n$ and $H$ is the mean curvature of $\Sigma_i$ in $(\Omega, g)$. This means that the mean curvature of $\Sigma_i \simeq \Sigma_i \times \{0\}$ equals $H$ in the $g_{u_i}$ metric. By Proposition \ref{prop: monontone}, we have that
\begin{eqnarray} \label{eq-monotone1}
\int_{\hat \Sigma_i} \hat H - \int_{\Sigma_i } H  &=& \int_{\hat \Sigma_i} \hat H - \int_{\Sigma_i \times \{0\}} H_{u_i} \\ \nonumber &=& \int_{{\Sigma_i \times \{0\}}} (H_\eta - H_{u_i} )  \geq
\int_{\Sigma_i \times \{1\}} (H_\eta - H_{u_i}) 
\end{eqnarray}
where as before, $H_\eta$ is the mean curvature of $\Sigma_i \times \{1\}$ in the $g_\eta$ metric (or, equivalently, the Euclidean mean curvature of $F(\Sigma_i, 1)$), and $H_{u_i}$ the mean curvature of $\Sigma_i \times \{1\}$ with respect to the metric $g_{u_i}$. 
For convenience, we omit writing the volume forms.

Next, let $N_i$ be the exterior region of the strictly convex hypersurface $F(\Sigma_i, 1)$ in $ \R^n $. By the result of Shi and Tam \cite{ShiTam02} (cf. Theorem \ref{thm: ShiTam}) there exists an asymptotically flat metric $g_i$ on $N_i$ 
such that $g_i $ has vanishing scalar curvature and such that along the hypersurface $F(\Sigma_i, 1)$ the metric $g_i$ coincides with the Euclidean metric, and its mean curvature with respect to $g_i$ equals the mean curvature of $\Sigma_i \times \{1\}$ with respect to $g_{u_i}$ via the identification $F(\cdot, 1)$. Furthermore, 
\begin{equation} \label{eq-monotone2}
\int_{\Sigma_i \times \{1\}}  (H_{\eta} - H_{u_i} ) = \int_{F(\Sigma_i, 1)} \hat H - \int_{\Sigma_i \times \{1\}} H_{u_i} 
\geq c(n) \mathfrak{m} ( {g}_i ), 
\end{equation}
where $ c(n) $ is a positive constant depending only on $ n $ and $ \mathfrak{m} ( {g}_i)$ is the ADM mass of $ {g}_i $ \cite{ADM61}. 

Now we have $1+2p$ Riemannian manifolds $(\Omega, g)$, $(\Sigma_i \times [0, 1], g_{u_i})$, $(N_i, g_i)$ with non-negative scalar curvature whose boundaries are identified via $\iota$ and $F(\cdot, 1)$ respectively such that the `inner and outer mean curvatures' match along these identifications. Glue these manifolds together to obtain an asymptotically flat manifold $(M, \mathcal{G})$ as in \cite{ShiTam02}. 
If $ \Omega $ is spin, then the positive mass theorem for spin manifolds with Lipschitz singularities in \cite[Section 3]{ShiTam02} can be applied to $ (M, \mathcal{G}) $ and shows that $\mathfrak{m}(g_i) \geq 0$, which together with (\ref{eq-monotone1}), (\ref{eq-monotone2}) gives inequality (\ref{eqn: main theorem}). Moreover, if equality in (\ref{eqn: main theorem}) holds, then the rigidity statement of the positive mass theorem for spin manifolds with Lipschitz singularities in \cite[Section 3]{ShiTam02} implies that $(\Omega, g)$ is isometric to a subset of $(\R^n, g)$. We note that if $3 \leq n\leq 7$ but $ \Omega $ is not assumed to be spin, then $\mathfrak{m}(g_i) \geq 0$ follows from the positive mass theorem for manifolds with corners along hypersurfaces in \cite{Miao02}. Hence inequality (\ref{eqn: main theorem}) remains valid provided $ 3 \le n \le 7$ even when $\Omega$ is not spin.

\section{Equality case in Proposition \ref{prop: main} without the spin assumption} \label{sec: rigidity}

In this section, we examine the equality case in Proposition \ref{prop: main} when $ \Omega $ is not assumed to be spin 
 but the dimension $ n $ satisfies $ 3 \le n \le 7$. We will use several ideas from \cite{SchoenYau79} related to the change of mass under conformal changes of the metric. The following lemma is standard and well-known to experts in geometric measure theory: 

\begin{lemma} \label{lem: GMT lemma} Let $(\Omega, g)$ be a compact Riemannian manifold of dimension $3 \leq n \leq 7$ with boundary components $\{\Sigma_1, \ldots, \Sigma_p\}$ where $p \geq 2$. Assume that the boundary is mean convex with respect to the outward pointing unit normal. Then for every $i \in \{1, \ldots p\}$ there exists a subdomain $\Omega_i \subset \Omega$ so that $\partial \Omega_i$ has two components: $\Sigma_i$ and a smooth embedded minimal surface. Put differently, every boundary component is homologous to a minimal surface in $(\Omega, g)$. 
\end{lemma}

We divide the proof of the equality case in Proposition \ref{prop: main} into a sequence of lemmas. 

\begin{lemma} \label{lem: ShiTam with minimal boundaries} 
Suppose  $3 \leq n \leq 7$. The inequality (\ref{eqn: main theorem}) in Proposition \ref{prop: main} remains true if $(\Omega, g)$ has additional boundary components, whose union we denote by $\Sigma_0,$ provided  $\Sigma_0$ is a minimal surface. 
\end{lemma}

\begin{proof}
We proceed as in the proof of inequality \eqref{eqn: main theorem} in the case where $\Sigma_0$ is empty by constructing an asymptotically flat manifold $(M, \mathcal{G})$ which contains $(\Omega, g)$ isometrically and has non-negative scalar curvature (distributionally across the boundary components $\Sigma_i$  for $i \in \{1, \ldots, p\}$). We double $(M, \mathcal{G})$ across the minimal surface $\Sigma_0$ and verify that the positive mass theorem applies in this situation, cf. \cite{Bray01}. For ease of notation, we continue to denote the doubled manifold by $(M, \mathcal{G})$. By \cite[Proposition 3.1]{Miao02}, there exists a sequence of smooth metrics $g_\delta$ such that $ g_\delta$ agrees with $\mathcal{G}$ outside a $\delta$-neighborhood of the union of $\bigcup_{i=1}^p \Sigma_i$, 
$ \bigcup_{i=1}^p F(\Sigma_i, 1)$,
their reflections, and $\Sigma_0$ such that $g_\delta \to \mathcal{G}$ in $\mathcal{C}^{0}$ as $ \delta \rightarrow 0$ and such that the scalar curvature $R(g_\delta)$ is bounded below by a constant that is independent of $\delta>0$. 
Let $ R(g_\delta)_{-} $ denote the negative part of $ R( g_\delta)$. As in \cite[Section 4.1]{Miao02}, 
we let $u_\delta : M \to \mathbb{R}$ be the (unique, positive) solution of
$$- \Delta_{g_\delta} u_\delta + \frac{n-2}{4(n-1)} R(g_\delta)_{-} u_\delta = 0,$$ 
which tends to $1$ in all of the $2p$ asymptotically flat ends.
Let $\mathcal{G}_\delta := u_\delta^{\frac{4}{n-2}} g_\delta$.
By \cite[Lemma 4.2]{Miao02}, $ \lim_{\delta \rightarrow 0+}
\mathfrak{m} ( \mathcal{ G}_\delta, N_i) = \mathfrak{m} ( \mathcal{G}, N_i ) $ 
for each end $N_i$ of $ M$. 
Since $ \mathcal{G}_\delta $ has non-negative scalar curvature, 
the positive mass theorem implies $ \mathfrak{m} ( \mathcal{ G}, N_i) \ge 0$. Therefore, $  \mathfrak{m} ( \mathcal{G} ) \ge 0 $, and the lemma follows as in the proof of Proposition \ref{prop: main} in Section \ref{pfofthm}.  
\end{proof}

\begin{lemma} \label{lma: R-zero-Omega_i}
Suppose $ 3 \le n \le 7$ and that the equality (\ref{eqn: main theorem}) in Proposition \ref{prop: main}  holds for some boundary component $\Sigma_i$, $i \in \{1, \ldots, p\}$. Then
$H_{\Sigma_i} \equiv \iota^*(\hat H_{\hat \Sigma_i})$. 
Moreover, if $ p \ge  2 $ and if  $ \Sigma_0 $ is the minimal surface  provided by Lemma \ref{lem: GMT lemma} such that  $ \Sigma_0 $ and $ \Sigma_i $ bound a subdomain
 $ \Omega_i $ of $\Omega$, then $R(g) \equiv 0$ on $ \Omega_i$. 
\begin{proof}
Let $ N_i $ be the asymptotically flat end corresponding to 
$ \Sigma_i $ that is constructed in  Section \ref{pfofthm}.
Then by assumption $ 0 \le \mathfrak{m} (N_i, \mathcal{G}) \leq \int_{\Sigma_i} (\hat H - H) = 0$.
By Proposition \ref{prop: monontone}, we  conclude that  
$ u_i = \eta_i $, where $ u_i $ and $ \eta_i $ are the functions in
the proof in Section \ref{pfofthm}. In particular, this implies 
$H_{\Sigma_i} \equiv \iota^*(\hat H_{\hat \Sigma_i})$.

Next suppose $ p \ge 2 $ and let $ \Sigma_0 $, $ \Omega_i$ be as in the statement of the lemma. Suppose $R(g)$ is strictly positive somewhere in $\Omega_i$. Let $u: \Omega_i \to \mathbb{R}$ be the solution of the Dirichlet problem $- \Delta_g u + \frac{n-2}{4(n-1)} R(g) u = 0$ such that $u = 1$ on $\partial \Omega_i$. Note that $u$ cannot be constant. The strong maximum principle gives that $0 < u \leq 1$, with equality only on the boundary, and further that $\nu(u)>0$ where $\nu$ is the outward pointing unit normal. It follows that in the conformal metric $\bar g := u^{\frac{4}{n-2}}g $ the boundary components $\Sigma_0$, $\Sigma_i$ have strictly greater mean curvature than with respect to the $g$ metric, so that in particular $\Sigma_0$ becomes mean convex in $(\Omega, \bar{g})$.
Now  let $ \bar{N}_i $ be the asymptotically flat end constructed in  Section \ref{pfofthm} corresponding to the boundary component 
$ \Sigma_i $ in $(\Omega_i, \bar{g})$. Let $(M, \bar{\mathcal{G}})$ be the asymptotically flat manifold obtained by gluing $(\Omega, \bar{g})$ and
$ \bar{N}_i$ along $ \Sigma_i$. 
Let $ \bar{H} $ be the mean curvature of $ \Sigma_i $ in $(\Omega_i, \bar{g})$.   By the proof in Section \ref{pfofthm} and the fact that
$ \bar{H} > H $,  we then have 
\begin{equation} \label{eq-A}
\mathfrak{m}( \bar {\mathcal{G}}, N_i)  \leq \int_{\Sigma_i} (\hat H - \bar H)  < 0.
\end{equation}
\noindent On the other hand, we can double the manifold $(M, \bar{\mathcal{G}})$ across its mean convex boundary $\Sigma_0$ just as in the proof of Lemma \ref{lem: ShiTam with minimal boundaries} to conclude that $\mathfrak{m}( \bar {\mathcal{G}}, N_i) \ge 0$. This is a contradiction to \eqref{eq-A}. Therefore, $ R (g)$ must be identically
zero in $ \Omega$, as asserted.
\end{proof}
\end{lemma}

Our next lemma is a minor modification of \cite[Corollary 2.1]{MiaoShiTam09}:

\begin{lemma} \label{lem: extension MiaoShiTam}
Let $(\Omega, g)$ be a smooth compact connected manifold with disconnected boundary $\partial \Omega = \Sigma_0 \dot \cup \Sigma$ such that $\Sigma$ is mean convex, $\Sigma_0$ is a minimal surface, and such that the scalar curvature $R(g) \equiv 0$. 

Suppose that $g$ is such that 
$\int_{\Sigma} \bar H$
is largest amongst all nearby (in $\mathcal{C}^2$) metrics $\bar g$ that are scalar flat, so that $\bar {g}|_{T \Sigma} = g|_{T \Sigma}$, and such that $\Sigma_0$ and $\Sigma$ are respectively minimal and mean convex (with mean curvature $\bar H$) in $(\Omega, \bar g)$. Then $Ric(g) \equiv 0$. 
\begin{proof}
Let $h$ be a smooth symmetric $(0, 2)$-tensor compactly supported in the interior of $\Omega$, let $g_t = g + t h$ for small $t$, 
and let $u_t$ be the unique positive solution of 
\begin{eqnarray*}
- \Delta_{g_t} u_t +  \frac{n-2}{4(n-1)} R(g_t) u_t &=& 0  \text{ on } \Omega \\ 
u_t &=& 1 \text{ on }\Sigma \\ 
\nu(u_t) &=& 0 \text{ on } \Sigma_0. 
\end{eqnarray*}
(The difference with the result in \cite{MiaoShiTam09} is the minimal boundary on which we are prescribing Neumann boundary data here.) Existence of 
such a $ u_t $ follows from the proof of (3.4) in \cite[Lemma 3.2]{SchoenYau79}.
That $ u_t $ is differentiable with respect to $t$ near $t = 0$ follows as in  \cite[pages 73-74]{SchoenYau79}.
Consider the metric $\bar g_t : = u_t^{\frac{4}{n-2}} g_t$. Then $\bar g_t$ competes with $g$ for least $\int_{\Sigma} \bar H$ and hence
\begin{eqnarray*}
0 = \frac{d}{dt}|_{t=0} \int_{\Sigma} \bar H_t =   \frac{2(n-1)}{n-2}  \frac{d}{dt}|_{t=0} \int_{\Sigma}\nu(u_t). 
\end{eqnarray*} 
Using the fact $ \nu(u_t) = 0  $ on $ \Sigma_0$,
integrating by parts and using the variation formula for scalar curvature $\frac{d}{dt}|_{t=0} R(g_t) = - \Delta \tr (h) + \div (\div h) - (h , Ric)$ this implies that
\begin{eqnarray*}
0 &=& \int_{\Omega} \frac{d}{dt}|_{t=0} \Delta_{g_t} u_t \\  &=& \frac{n-2}{4(n-1)} \int_\Omega - \Delta \tr (h) + \div (\div h) - (h, Ric)  \\ &=& - \frac{n-2}{4(n-1)} \int_{\Omega} (h, Ric). 
\end{eqnarray*}
Since this is true for all directions $h$ as above, this implies that $Ric \equiv 0$, as desired.
\end{proof}
\end{lemma}

Finally, we need the following lemma, which is a slight extension of 
 \cite[Proposition 1]{HangWang07}. 
  
\begin{lemma} \label{lem: variant of HangWang}
Let $(\Omega, g)$ be a smooth compact connected Riemannian manifold of dimension $n \ge 2$ with smooth boundary $\partial \Omega = \Sigma_0 \dot \cup \Sigma$. We assume that $\Sigma_0$ is either empty or else has non-negative mean curvature with respect to the outward pointing unit normal. We assume that $\Sigma$ is non-empty and connected, and that there exists an isometric embedding $$ \iota : \Sigma \rightarrow \R^n $$ such that $\hat \Sigma:= \iota(\Sigma)$ is a closed embedded hypersurface of $\mathbb{R}^n$, and such that  $$ H \geq |\iota^*\hat H|.  $$ Here, $ H $ and $ \hat H$  denote the respective mean curvatures of $ \Sigma \subset \Omega$ and $ \hat \Sigma \subset \R^n$ with respect to the outward normals. If  $(\Omega, g)$ has non-negative Ricci curvature, then $\Sigma_0 = \emptyset$ and $(\Omega, g)$  is isometric to the (compact) domain bounded by $ \hat \Sigma $ in $ \R^n$.
\begin{proof} We only discuss the parts of the proof which differ slightly from \cite{HangWang07}. Consider the harmonic functions $X_i : \Omega \to \mathbb{R}$ with boundary values $\iota^*(x_i)$ on $\Sigma$, where $(x_1, \ldots, x_n)$ are coordinates on $\R^n$, and boundary value $0$ on $\Sigma_0$. As in \cite{HangWang07}, one computes using Reilly's B\^ochner formula that 

\begin{eqnarray*} \sum_{i = 1}^n \int_\Omega | \nabla^2 X^i |^2 &+& \sum_{i=1}^{n} \int_{\Sigma_0} H_{\Sigma_0} \nu(X^i)^2 \\ 
& \leq & -  \int_\Sigma  H \lf\{ 1 -  2  |( d X ( \nu ) \cdot \hat \nu )|  +  | d X ( \nu ) |^2 \ri\} \nonumber \\
& \leq &  -  \int_\Sigma H  \lf\{
\lf[  1 -   |d X ( \nu ) \cdot \hat \nu|  \ri]^2  +  \lf[ | d X ( \nu ) |^2 -  ( d X ( \nu ) \cdot \hat \nu )^2 
\ri] \ri\}
\nonumber \\
& \leq & 0 
\end{eqnarray*}
where the map $X: \Omega \to \R^n$ is defined as $X = (X_1, \ldots, X_n)$, where $dX$ is its tangent map, and where $\hat \nu$ is the unit normal vector pointing out of $\hat \Sigma \subset \R^n$. 
Since it is assumed that $H_{\Sigma_0}\geq0$ it follows as in \cite{HangWang07} that
\be \label{eq-parallelX1}
\nabla^2 X^i  = 0 \ \mathrm{on} \ \Omega, \ \forall i = 1, \ldots, n
\ee
and, since $H \geq |\hat H|$ is greater than zero at least at one point, 
\be \label{eq-parallelX2}
 |d X ( \nu ) \cdot \hat \nu|  = 1, 
\ \  
| d X ( \nu ) |^2  =  ( d X ( \nu ) \cdot \hat \nu )^2  \ \text { at some point  } \ p \in  \Sigma,
\ee
\noindent which implies readily that $g \equiv\sum_{i, j=1}^n \delta_{ij} d X^i \otimes d X^j \equiv X^*(\delta_{ij})$ on $\Omega$. 
Note that since $n \geq 2$ and $X_i|_{\Sigma_0} = 0$, this is only possible if $\Sigma_0 = \emptyset$. 
In that case, it follows that $X : (\Omega, g) \to (\mathbb{R}^n, \delta)$ is a local isometry. In particular, a collar neighbourhood of $\partial \Omega = \Sigma$ is isometric to a collar neighbourhood of $\iota (\Sigma) = \hat \Sigma \subset \R^n$ so that $\Omega$ can be glued smoothly to $\R^n$. Since $\Omega$ is flat, it follows that $\Omega$ is a domain in $\R^n$. 

\end{proof}
\end{lemma}

We are now in a position to prove the rigidity statement in Proposition \ref{prop: main} without the spin assumption: Suppose that $\int_{\Sigma_i} H= \int_{\hat \Sigma_i} \hat H$ for some $i \in \{1, \ldots, p \}$. If $p \geq 2$, let $\Omega_i \subset \Omega$ be a domain as in Lemma \ref{lem: GMT lemma} such that $\partial \Omega_i = \Sigma_0 \dot \cup \Sigma_i$ where $\Sigma_0$ is a minimal surface. 
%%%%%%%%%%%%%%%%%%%%%%
%
We can use Lemma \ref{lem: ShiTam with minimal boundaries} to
justify the use of Lemma \ref{lma: R-zero-Omega_i} and 
Lemma \ref{lem: extension MiaoShiTam} to conclude that 
$H = \hat H$ on $\Sigma_i$, $R(g) \equiv 0$ on $\Omega_i$, and in fact
$\Ric(g) \equiv 0$ on $ \Omega_i$.
By Lemma \ref{lem: variant of HangWang} we conclude that $\Sigma_0 = \emptyset$. So there could only have been one boundary component $\Sigma_1$ in the first place. Then Lemmas \ref{lem: extension MiaoShiTam} and \ref{lem: variant of HangWang} show that $(\Omega, g)$ is isometric to a domain in $ \R^n $.  

\begin{remark}
In the above approach, we followed the idea 
in \cite[Section 3]{SchoenYau79} to show that, if equality in  \eqref{prop: main} holds, then  $  \Omega $ must have a single boundary component.  On the other hand,  the presence of a minimal surface in Lemma \ref{lem: GMT lemma} also suggests that one can apply  the higher dimensional Riemannian Penrose Inequality \cite{BrayLee09}  to directly prove that the strict inequality in   \eqref{prop: main}  must hold whenever  $ \p \Omega $ has more than one components. This second approach can be made rigorous by the arguments in \cite[Section 3.2]{Miao08}. 
\end{remark}

\section{Results of Gerhardt and Urbas} \label{Gerhardt-Urbas}

To finish the proof of Theorem \ref{thm: main theorem}, we recall the following result on expanding star-shaped surfaces in $ \R^{n} $ into spheres, which was obtained by
Gerhardt \cite{Gerhardt90} and  Urbas \cite{Urbas90}. 

\begin{theorem} {\rm (Gerhardt, Urbas)} \label{thm: Gerhardt and Urbas}
Let ${\hat \Sigma} $ be a smooth, closed, compact hypersurface in $ \R^{n}$, given by 
a smooth embedding $\iota : \Sigma^{n-1} \rightarrow \R^{n} $, and suppose that $ {\hat \Sigma} $ is star-shaped
with respect to a point $P_0 \in \R^n$.
Let $ \Gamma \subset \R^{n-1} $ be an open, convex, symmetric cone with vertex at the origin,
which contains the positive cone $ \Gamma^+ = \{ (\lambda_1, \ldots, \lambda_{n-1} ) \ | \ 
\lambda_i > 0, \ \forall i \}$. Let $ f \in \mathcal{C}^\infty( \Gamma ) \cap \mathcal{C}^0( \bar{\Gamma} ) $ 
be a symmetric function satisfying 
\begin{enumerate} 
\item $ f $ is homogeneous of degree one 
\item $ \frac{ \p f}{\p \lambda_i } > 0 $ and $ f $ is concave 
\item $ f> 0$ on $\Gamma$ and $f \equiv 0 $ on $ \p \Gamma $.
\end{enumerate} 
Suppose $ f( \kappa_1 ( \sigma), \ldots, \kappa_{n-1} ( \sigma) ) > 0$,
where $ \{ \kappa_i \}_{i=1}^{n-1}$ are the principal curvatures 
of $ {\hat \Sigma} $ at  $ \sigma \in {\hat \Sigma} $. Then the initial value problem
\begin{equation} \label{eq-initial} 
\left\{
\begin{split}
\frac{ \p F}{ \p t}  = & \ \frac{1}{ f(\sigma, t) } \nu(\sigma, t)  \\
F( \sigma, 0)  = & \iota (\sigma)
\end{split} 
\right.
\end{equation}
has a unique smooth solution $ F : \Sigma^{n-1} \times [0, \infty) \rightarrow \R^{n} $, where
$ \nu( \sigma, t) $ is the outer unit normal vector field to $F(\Sigma, t)$ at $F(\sigma, t)$, and
$$ f(\sigma, t) = { f ( \kappa_1, \ldots, \kappa_{n-1} ) } ,$$
where $ \{ \kappa_i \} $  are the principal curvatures of $F(\Sigma, t)$ at $F(\sigma,t)$. 
Moreover, for each $ t$, $F(\Sigma, t)$ is a star-shaped surface with respect to $P_0$, and
$ e^{-t} F(\Sigma, t) $ converges to a sphere centered at $ P_0 $ in the $ \mathcal{C}^\infty$ topology
as $ t \rightarrow \infty$. 
\end{theorem}

Now suppose $ {\hat \Sigma} = \iota (\Sigma) \subset \R^n$ is a smooth closed embedded hypersurface, star-shaped and such that
\begin{equation}
H = \sum_{i=1}^{n-1} \kappa_i  > 0 , \ \ 
2 K  = ( \sum_{i=1}^{n-1} \kappa_i )^2 - \sum_{i=1}^{n-1} \kappa_i^2 > 0,
\end{equation}
i.e., the surface has positive mean and scalar curvatures. Let  $ \Gamma \subset \R^{n-1} $ be the open, convex, connected symmetric cone
\begin{equation}
\Gamma = \left\{ ( \lambda_1, \ldots, \lambda_{n-1} ) \ | \ 
\sum_{i=1}^{n-1} \lambda_i > 0 \ \mathrm{and} \ 
( \sum_{i=1}^{n-1} \lambda_i )^2 - \sum_{i=1}^{n-1} \lambda_i ^2 > 0 
\right\}.
\end{equation}
On $ \Gamma $, let $ f $ be the smooth, positive function defined by
\begin{equation} \label{eq-ourf}
f ( \lambda_1, \ldots, \lambda_{n-1} ) =  \frac{ ( \sum_{i=1}^{n-1} \lambda_i )^2 
- \sum_{i=1}^{n-1} \lambda_i ^2 }{ \sum_{i=1}^{n-1} \lambda_i  }  .
\end{equation}
It is easily seen that $ f $ can be continuously extended
to $ \bar{\Gamma} $ such that $ f \equiv 0 $ on $ \p \Gamma $. We leave it as an entertaining exercise for the reader to verify that $f$ satisfies the conditions in Theorem $\ref{thm: Gerhardt and Urbas}$. (See the introduction of \cite{Urbas90} for reference on the application of the theorem for symmetric functions of the principle curvatures.) It follows that there exists a smooth map $ F: \Sigma \times [0, \infty) \rightarrow \R^n $ such that $ F( {\Sigma}, 0) = {\hat \Sigma} $,
 \begin{equation}
  \frac{\p F}{\p t} = \frac{ \sum_{i=1}^{n-1} \kappa_i  } { ( \sum_{i=1}^{n-1} \kappa_i )^2 - \sum_{i=1}^{n-1} \kappa_i ^2 }  \nu = \frac{H}{R} \nu
  \end{equation}
where $ \nu  $ is the outer unit normal to the closed hypersurface $F( {\hat \Sigma}, t)$,  $\{ \kappa_i \}$ are its principal curvatures, $H$ its mean curvature, and $R$ its scalar curvature. Moreover, 
$e^{-t} F(\Sigma, t) $ converges to a sphere  in the $ \mathcal{C}^\infty$ topology as $ t \rightarrow \infty$. In particular it follows that  $F(\Sigma, t)$ has positive mean and scalar curvatures throughout the evolution, and that it is a strictly convex hypersurface for $t$ sufficiently large. 

\begin{remark}
An obvious approach to proving Theorem \ref{thm: main theorem} is to use this $\frac{H}{R}$-flow of Gerhardt and Urbas to connect $\hat \Sigma_i$ to `round spheres at infinity' and to carry out the program of Shi and Tam along this foliation without using their distance surface foliation once $\hat \Sigma_i$ is deformed to a strictly convex set. While this approach might have some independent interest, it would require detailed knowledge of the asymptotics of the $\frac{H}{R}$-flow to ensure that the resulting metric is indeed asymptotically flat, and that the Brown-York mass of the leaves converges to its mass.
\end{remark}

\bibliographystyle{amsplain}

\end{document}